\documentclass{amsart}
\usepackage{graphicx} 

\usepackage{amsfonts, amsthm, xcolor,hyperref}
\usepackage[all,cmtip]{xy}
\usepackage{mathtools}
\usepackage{enumitem}

\newcommand{\PP}{\mathbb{P}}
\newtheorem{thm}{Theorem}[section]
\newtheorem{prop}[thm]{Proposition}
\newtheorem{lem}[thm]{Lemma}
\newtheorem{cor}[thm]{Corollary}
\newtheorem{defn}[thm]{Definition}

\newtheorem{rem}[thm]{Remark}
\newtheorem{exmp}[thm]{Example}

\newcommand{\Deltatilde}{\tilde{\Delta}}

\DeclareMathOperator{\Sing}{Sing}

\DeclareMathOperator{\pic}{Pic}

\newcommand{\piDoubleCover}{f}
\newcommand{\XDoubleCover}{S}
\newcommand{\piDelta}{\varpi}
\newcommand{\tto}{\dashrightarrow}

\newcommand{\defi}[1]{\textsf{#1}}

\definecolor{ao(english)}{rgb}{0.0, 0.5, 0.0}

\title[Multisections of conic bundles]{Nonexistence of degree two rational multisections of conic bundles over the plane}

\author{Jeffrey Diller}
\address{Department of Mathematics, University of Notre Dame, 255 Hurley Hall, Notre Dame, IN 46556}
\email{diller.1@nd.edu}

\author{Lena Ji}
\address{Department of Mathematics, University of Illinois Urbana-Champaign, 214 Harker Hall, 1305 W. Green Street, Urbana, IL 61801}
\email{lenaji.math@gmail.com}

\author{Eric Riedl}
\address{Department of Mathematics, University of Notre Dame, 255 Hurley Hall, Notre Dame, IN 46556}
\email{eriedl@nd.edu}

\thanks{During the preparation of this article, J.D. was supported by NSF grant DMS 2246893, L.J. was supported by NSF grant DMS-2501990 and by Simons Foundation gift 00013959, and E. R. was supported by NSF CAREER grant DMS-1945944 and Simons Foundation grants 00011850 and 00013673.}

\begin{document}

\begin{abstract}
    We prove that a standard conic bundle $X \to \mathbb P^2_{\mathbb C}$ whose discriminant is very general of degree at least 18 admits no rational multisections of degree two. This is the first step towards proving a conjecture of Iskovskikh that there are conic bundle threefolds that are not unirational, since to prove that $X$ is not unirational, it suffices to show that there are no rational multisections of any degree. Proving Iskovskikh's conjecture would provide the first example of a rationally connected variety that is not unirational.
\end{abstract}

\maketitle

\section{Introduction}
There are many related notions of what it means for a projective variety $X$ to be similar to projective space. For instance, $X$ is rational if it is birational to projective space; unirational if there is a dominant rational map from projective space onto $X$; and rationally connected if two general points in $X$ can be joined by rational curves. It is easily seen that rational varieties are unirational, and that unirational varieties are rationally connected. Celebrated results from the 70's \cite{ClemensGriffiths, ArtinMumford, IskovskikhManin} give examples of unirational varieties that are not rational. However, it remains a longstanding open question whether there are smooth rationally connected varieties that are not unirational.

Conic bundles over \(\mathbb P^2\) form an important class of rationally connected varieties.
The total space of a standard conic bundle is known to be rational if the discriminant curve has degree \(\leq 4\);
furthermore, when the discriminant has degree \(\leq 8\), the total space is known to be unirational \cite{KollarMella}.
Moreover, there exist unirational examples of standard conic bundles over \(\mathbb P^2\) with discriminant of arbitrarily high degree \cite{MassarentiMella24} (see \cite{BCTSSD} for examples where the base of the conic bundle is a Hirzebruch surface).
However, when the discriminant curve has sufficiently high degree \emph{and} is sufficiently general, Iskovskikh conjectured that the total space is not unirational \cite{Iskovskikh79}. The goal of this paper is to offer some evidence in the direction of this conjecture.

A unirationality criterion due to Enriques states that the total space of a conic bundle \(\pi\colon X\to\mathbb P^2\) is unirational if and only if \(\pi\) admits a rational multisection of some degree. If $\pi$ admits an odd degree multisection, then Springer's theorem implies that $\pi$ admits a rational section, which implies that $X$ is rational.  So only even degree multisections can exist if \(X\) is not rational, and the rationality problem for conic bundle threefolds has been extensively studied (see \cite{Prokhorov-survey} for a general discussion).  Our main result is a nonexistence result for the first even degree:

\begin{thm}\label{thm:main-thm}
    Let \(\pi\colon X\to\mathbb P^2_{\mathbb C}\) be a standard conic bundle. If the discriminant curve of \(\pi\) has degree \(\geq 18\) and is very general, then \(\pi\) does not admit a rational multisection of degree \(2\).
\end{thm}

It remains an open question whether the conic bundle threefolds in Theorem~\ref{thm:main-thm} admit unirational parametrizations of degree two (see Remark~\ref{rem:rational-multisection-degree}).

For any conic bundle \(\pi\colon X\to\mathbb P^2\), there exist (not necessarily rational) surfaces \(S\subset X\) such that \(\pi|_S\colon S\to\mathbb P^2\) is generically finite of degree two.
Previously, Koll\'ar studied nonexistence of certain such \(S\) with low degree branch curve, and used this to prove that many conic bundle threefolds are not birationally equivalent to the underlying variety of a numerical Calabi--Yau pair \cite{Kollar17}.

There has been previous work studying unirationality questions for other rationally connected varieties. Pukhlikov has several papers showing the nonexistence of abelian Galois unirational parametrizations of certain Fano varieties and Mori fiber spaces \cite{Pukhlikov20,Pukhlikov21,Pukhlikov26}.
In another direction,
Koll\'ar observed that a unirational variety must contain rational surfaces through a general point, and
there have been several works restricting the rational surfaces on certain Fano hypersurfaces \cite{BeheshtiFewrationalSurfaces, BeheshtiRiedlPositivityResults, BeheshtiStarrRationalSurfaces, BeheshtiRiedlRestrictions}.

The structure of our paper is as follows. In Section~\ref{sec:conicBundles} we discuss the geometry of conic bundles $X \to \PP^2$ and give a criterion for when a generically finite morphism $S \to \PP^2$ has pullback $X_S\to S$ with a rational section. In Section~\ref{sec:doubleCovers}, we discuss previous work about rational double covers of $\PP^2$ and describe their classification. In Section~\ref{sec:hyperbolicity}, we discuss how algebraic hyperbolicity results for the complement of plane curves allow us to prove Theorem~\ref{thm:main-thm}.

Throughout we work over \(\mathbb C\). By a surface, we mean an integral separated scheme of finite type over \(\mathbb C\). We do not require curves to be irreducible or reduced.

\subsection{Acknowledgments}
We thank A. Calabri, C. Ciliberto, B. Hassett, J. Koll\'ar, B. Lehmann, A. Patel, A. Pukhlikov, V. V. Shokurov, and S. Tanimoto for helpful discussions.

\subsection{AI disclosure}
After completing an initial draft of this paper, we found a technical obstacle in showing a bound of the type in Section~\ref{sec:doubleCovers}. As we looked for solutions, we consulted with ChatGPT-5.6 Sol and it made a suggestion that is the basis for our proof of Proposition~\ref{prop:5deg-bound}. This paper was written by its authors, who have independently checked all mathematical assertions in the paper, and we take full responsibility for its correctness.

\section{Conic bundles} \label{sec:conicBundles}

In this section, we recall definitions and preliminary results about conic bundles.
For a reference on conic bundles, see, e.g., \cite{Prokhorov-survey}.

\begin{defn}
    Let \(S\) be a projective rational surface.
    A \defi{conic bundle} \(\pi\colon X\to S\) is a proper flat morphism such that every fiber is isomorphic to a plane conic and the generic fiber is smooth.
    If \(X\) and \(S\) are smooth, the \defi{discriminant} \(\Delta\subset S\) of a conic bundle \(\pi\) is the divisor parametrizing the singular fibers of \(\pi\), and
    the \defi{discriminant double cover} \(\piDelta\colon\Deltatilde\to\Delta\) is the Stein factorization of the normalization of \(\pi^{-1}(\Delta)\).
    A conic bundle \(\pi\colon X\to S\) is \defi{standard} if \(X\) and \(S\) are smooth and the relative Picard rank \(\rho(X/S)\) is \(1\).
\end{defn}
If \(\pi\) is standard and \(\Delta\) is smooth, then \(\piDelta\colon\Deltatilde\to\Delta\) is a nontrivial \'etale double cover of degree two.

Examples of standard conic bundles are plentiful.
For every smooth plane curve $\Delta$ of positive genus, there exists standard conic bundles of discriminant \(\Delta\) by \cite[\S3]{ArtinMumford} (see also \cite[Proposition 3.10]{Prokhorov-survey}).

A criterion due to Enriques states that for a conic bundle over a rational surface, the total space is unirational if and only if the conic bundle admits a rational multisection:

\begin{prop}[{Enriques criterion, see \cite[Proposition 10.1.1]{IskovskikhProkhorov}}]\label{prop:unirationality-criterion-Enriques}
    Let \(\pi\colon X\to\mathbb P^2\) be a conic bundle. Then \(X\) is unirational if and only if there is a projective rational surface \(S\) and a generically finite morphism \(S\to\mathbb P^2\) such that the base change \(\pi_S\colon X_S\to S\) admits a rational section.
\end{prop}

By Springer's theorem \cite{Springer52}, \(\pi\) admits a rational multisection of odd degree if and only if it admits a rational section, and this is furthermore equivalent to the condition that \(X\) is birationally equivalent over \(\mathbb P^2\) to the product \(\mathbb P^1\times\mathbb P^2\).

\begin{rem}\label{rem:rational-multisection-degree}
    A rational multisection of \(\pi\) yields a degree \(d\) dominant rational map \(\mathbb P^3\dashrightarrow X\) of the same degree.
    However, the forward direction of the Enriques criterion does not immediately give any conditions on the degree of a rational multisection: the existence of a degree \(d\) unirational parametrization of \(X\) does not imply the existence of a degree \(d\) rational multisection of the conic bundle \(\pi\).
    
    For instance, if \(\pi\colon X\to \mathbb P^2\) is a standard conic bundle with discriminant of degree 3 or 4, then \(X\) is rational and the map \(\pi\) does not have a rational section (see \cite[Lemma 3.7(iii) and Corollary 5.6.1]{Prokhorov-survey}).
    For another example with \(d=3\), recall that the projection of a smooth cubic threefold from a line has the structure of a standard conic bundle over \(\mathbb P^2\), and this always admits a degree two rational multisection. For certain smooth cubic threefolds, \cite{YangYuZhu25} recently constructed degree \(3\) unirational parametrizations, and the associated conic bundle does not admit a degree three rational multisection by \cite{ClemensGriffiths}.
\end{rem}

The following necessary criterion for a rational multisection is well known to experts (see, e.g., \cite[Proposition 18]{Kollar17}). We include a proof here for completeness.

\begin{prop}\label{prop:multisection-factors-through-Deltatilde}
    Let \(\pi\colon X\to \mathbb P^2\) be a standard conic bundle with irreducible discriminant \(\Delta\).
    Let \(\piDoubleCover\colon\XDoubleCover\to\mathbb P^2\) be a generically finite morphism from a normal projective surface, and assume the branch locus of \(\piDoubleCover\) does not contain \(\Delta\).
    Let \(\Delta_{\XDoubleCover}\) denote the closure of the preimage of \(\Delta\) over the locus where \(f\) is finite, and let \(\Delta_{\XDoubleCover}^\nu\) denote its normalization.
    If the base change \(\pi_{\XDoubleCover}\colon X_{\XDoubleCover}\to\XDoubleCover\) admits a rational section, then the map \(\Delta_{\XDoubleCover}^\nu\to\Delta\) induced by \(f\) factors through the degree two map \(\piDelta\colon \Deltatilde\to\Delta\).
\end{prop}

\begin{proof}
    Let \(\varphi\colon S\dashrightarrow X_S\) be a rational section.
    Since \(S\) is normal, \(\varphi\) is defined at the generic point of every irreducible component of \(f^{-1}(\Delta)\) and in particular induces a rational map \(\varphi|_{\Delta_\XDoubleCover}\colon \Delta_\XDoubleCover\dashrightarrow X_\XDoubleCover\) with image contained in \(\pi_\XDoubleCover^{-1}(\Delta_{\XDoubleCover})\).
    Since \(\varphi\) is a rational section and since the branch locus of \(\piDoubleCover\) does not contain \(\Delta\), the image of each component of \(\Delta_\XDoubleCover\) under \(\varphi|_{\Delta_\XDoubleCover}\) is not contained in the singular locus of \(\pi_\XDoubleCover^{-1}(\Delta_{\XDoubleCover})\).
    Therefore the normalization of \(\Delta_\XDoubleCover\) maps to the normalization of \(\pi_\XDoubleCover^{-1}(\Delta_{\XDoubleCover})\), which then maps to the normalization of \(\pi^{-1}(\Delta)\). Since the diagram below commutes, the map \(\Delta_\XDoubleCover^\nu\to\Delta\) induced by \(\piDoubleCover\) factors through \(\piDelta\colon\Deltatilde\to\Delta\) as claimed.
    \[\xymatrixrowsep{1pc}\xymatrix{
    & \pi_{\XDoubleCover}^{-1}(\Delta_{\XDoubleCover})^\nu \ar[r] \ar[dd]^{\pi_{\XDoubleCover}} & \pi^{-1}(\Delta)^\nu \ar[dd]^{\pi} \ar[rd] \\
    & & & \Deltatilde \ar[ld]_-{\piDelta} \\
    \Delta_{\XDoubleCover}^\nu \ar[ruu]^-{\varphi|_{\Delta_\XDoubleCover}} \ar[r]^-\nu & \Delta_{\XDoubleCover} \ar[r]^-{\piDoubleCover|_{\Delta_{\XDoubleCover}}} & \Delta
    } \]
\end{proof}

\begin{cor}\label{cor:deg-2-multisection-restriction}
    In the setting of Proposition~\ref{prop:multisection-factors-through-Deltatilde}, assume additionally that \(\Delta\) is smooth and \(f\) is finite of degree \(2\). If \(\pi_S\) admits a rational section, then the branch curve $B$ of $f$ does not meet \(\Delta\) transversely at any smooth point of $B$.
\end{cor}

\section{Rational double covers of \texorpdfstring{$\PP^2$}{P2}} \label{sec:doubleCovers}

Here we restrict attention to the case of \defi{double covers}, i.e.,
the case of finite surjective morphisms of degree two $f\colon S\to \PP^2$ from a projective normal surface.
Since \(S\) is normal, the branch divisor $B\subset\PP^2$ of $f$ is necessarily reduced.  It is moreover smooth if and only if $S$ is.  To handle singular $S$, we devote this section to obtaining the following bound.

\begin{prop}\label{prop:5deg-bound}
Let $f:S\to\PP^2$ be a double cover of $\PP^2$ by a (normal) rational surface $S$.  Let $x_1,\dots,x_k\in\PP^2$ be the singular points of the branch curve $B$ for $f$, and $c_1,\dots,c_k\geq 2$ denote the multiplicities of $B$ at these points.  Then 
    \[ 5\deg B - \sum_{i=1}^k c_i - 2g(B) \geq 8  \]
where $g(B)$ denotes the `total genus' of $B$, i.e., the sum of the geometric genera of the irreducible components of $B$.
\end{prop}

In order to present the proof we explain how (see \cite[\S 3]{Calabri}) the given double cover $f\colon S \to \PP^2$ admits a `canonical resolution'.  If $\rho\colon T\to \PP^2$ is a birational morphism from a smooth surface, then $f$ induces a double cover $\hat f\colon \hat S\to T$ of $T$ by a uniquely determined normal surface $\hat S$.  The branch curve $\hat B \subset T$ is again reduced, of the form $\tilde B + P$, where $\tilde B$ is the strict transform of $B$, and the support of $P$ is contained in the rational curves contracted by $\rho$.  The normal surface $\hat S$ is smooth if and only if $\hat B$ is, i.e., if and only if $\tilde B$ is the desingularization of $B$ and the irreducible components of $P$ are disjoint both from $\tilde B$ and from each other.  We call $\rho\colon T\to \PP^2$ a \defi{resolution} of $f$ when this happens.  

Regardless, let us factor $\rho$ into a sequence of point blowups
$$
T = T_\ell \overset{\sigma_\ell}{\to}T_{\ell-1} \overset{\sigma_{\ell-1}}{\to} \dots \overset{\sigma_1}{\to} T_0 = \PP^2.
$$
Let $x_i\in T_{i-1}$ denote the center of $\sigma_i$.  Let $E_i = \sigma_i^{-1}(x_i)\subset T_i$ denote the contracted curve as well as its strict transform by subsequent blowups.
A key fact is that if $B_i$ denotes the branch curve of the induced double cover $f_i\colon S_i\to T_i$, then $B_{i+1}$ is obtained by reducing the coefficient of $E_{i+1}$ in $\sigma_{i+1}^* B_i$ modulo $2$.  That is, if $E_{i+1}$ has odd multiplicity in $\sigma_{i+1}^* B_i$, then $B_{i+1} = \sigma_{i+1}^{-1}(B_i)$ is the set-theoretic total transform of $B_i$; and if $E_{i+1}$ has even multiplicity, then $B_{i+1} = \sigma_{i+1}^{-1}(B_i) - E_{i+1}$.  We call $E_{i+1}$ a \defi{virtual branch component} in the odd multiplicity case.
To get the \defi{canonical resolution of \(f\)}, which is uniquely determined up to isomorphism, one proceeds inductively, centering each blowup $\sigma_i$ at some singular point $x_i$ of $B_{i-1}$ until no such point remains.  From now on, we assume that $\rho$ is the canonical resolution of $f$.  The branch curve for the double cover $\hat f\colon \hat S \to T$ is then $\hat B = \tilde B +P$, where $P$ is the sum of the virtual branch components $E_i$.  For convenience, we order the blowups $\sigma_i$ so that $x_1,\dots,x_k\in \PP^2$ are the singular points of $B$.

We will say that a center $x_j$ is \defi{proximate} to a previous center $x_i$ if $x_j \in E_i$ (meaning the strict transform of $E_i$ in $T_{j-1}$) and \defi{immediately proximate} to $x_i$ if additionally $x_j$ is not proximate to $x_h$ for any $h>i$.  So while $x_j$ can be proximate to two distinct centers $x_i$, $i<j$, it can be immediately proximate to only one of them.  

To prove Proposition \ref{prop:5deg-bound}, we will compare the branch curve $\hat B\subset T$ with the total transform
$$
\rho^* B = \tilde B + \sum c_i E_i^*.
$$
Here $E_i^* = (\sigma_\ell\circ \dots \circ\sigma_{i+1})^* E_i$ is the total transform in $T$ of $E_i\subset T_i$, and $c_i$ denotes the multiplicity at $x_i\in T_{i-1}$ of the strict transform of $B$ by $\sigma_{i-1}\circ \dots\circ \sigma_1$.  In particular, $c_1,\dots,c_k$ are just the multiplicities in the statement of Proposition \ref{prop:5deg-bound}.

\begin{lem}
\label{lem:alternatives}
If $E_i$ is a virtual branch component of the canonical resolution of $f$, then one or both of the following are true.  
\begin{enumerate}
\item\label{item:alternative-1} There are at least two $j>i$ such that $x_j$ is proximate to $x_i$.
\item\label{item:alternative-2} There exists $j>i$ such that $x_j$ is immediately proximate to $x_i$ and $c_j\geq 2$. 
\end{enumerate}
\end{lem}

\begin{proof}
The center $x_i$ lies on at most two virtual branch components $E_h$ for $h<i$. If there are exactly two such, then $E_i$ will intersect (the strict transforms in $T_i$ of) each at a different point.  But in the canonical resolution, $E_i$ must be disjoint from both these components.  Hence both intersections must be blown up to reach $T$, and we are in Case~\eqref{item:alternative-1} of the lemma.  

If $x_i$ lies on at most one virtual branch component, then $x_i \in \tilde B_{i-1}$, where $\tilde B_{i-1}$ is the strict transform of $B$ in $T_{i-1}$.  Moreover we claim that $x_i$ must be a singular point of $\tilde B_{i-1}$. Indeed, if $x_i$ is a regular point of $\tilde B_{i-1}$, then the multiplicity of $B_{i-1}$ at $x_i$ is either $1$ (contradicting that $x_i$ is a singular point of $B_{i-1}$) or $2$ (contradicting that $E_i$ is a virtual branch component).  Thus counting multiplicity, $\tilde B_i$ will meet $E_i$ at least twice, and at least one of the following occurs:
\begin{itemize}
    \item $\tilde B_i\cap E_i$ contains at least two distinct points;
    \item $\tilde B_i$ is tangent to $E_i$; or
    \item $\tilde B_i$ is singular where it meets $E_i$.
\end{itemize}
In the first two cases, it is necessary to blow up at least twice more at points of $E_i$ to separate $\tilde B_i$ from $E_i$, so we are in Case~\eqref{item:alternative-1} of the lemma.  In the last case, it is necessary to blow up the point $x_j = E_i\cap \tilde B_i$ which is immediately proximate to $x_i$.  Since $\tilde B_i$ is singular at $x_j$, we have $c_j\geq 2$, and we are in Case~\eqref{item:alternative-2} of the lemma.
\end{proof}

\begin{proof}[Proof of Proposition \ref{prop:5deg-bound}]
We freely use the notation from the above discussion about the canonical resolution of $f$.  The following further facts about canonical resolutions are established in \cite[\S3,\S4]{Calabri}.
\begin{itemize}
    \item $\hat B = 2A$ is divisible by $2$ in $\pic(T)$. 
    \item If $\hat S$ (equivalently $S$) is rational, then the arithmetic genus of $A$ vanishes, i.e., $A\cdot (A+K_T) = -2$.
\end{itemize}
These imply that 
\begin{equation*}
    \begin{split}
        -8 & = 4A\cdot(A+K_T) = \hat B\cdot(\hat B+2K_T) = \tilde B\cdot(\tilde B+2K_T) + P\cdot(P+2K_T) \\
        & = \tilde B\cdot(\tilde B + 2K_T) - P^2 + 2P\cdot(P+K_T) \\
        & = 2(g(B) +1 -s) -2 + \tilde B\cdot K_T - \sum_{E_i\subset P} (E_i^2 + 4), 
    \end{split}
\end{equation*}
where $s$ is the number of irreducible components of $\tilde B$ (hence also of $B$).
The third equality follows from the fact that $\tilde B$ and $P$ are disjoint; the fifth follows from adjunction, the fact that $\tilde B$ is a desingularization of $B$, and the fact that \(P\) is a disjoint union of some of the rational curves contracted by \(\rho\).

Since
$
K_T = \rho^* K_{\PP^2} + \sum_{i=1}^\ell E_i^*
$, \(E_i^*\cdot E_j^*= - \delta_{ij}\), and \(s\leq \deg B\),
we may continue our computation as follows:
\begin{equation*}
    \begin{split}
        -8 & = 2(g(B)-s) -3\deg B + \sum_{i=1}^\ell c_i - \sum_{E_i\subset P} (4+E_i^2) \\
        & \geq 2 g(B) - 5\deg B + \sum_{i=1}^k c_i + \sum_{i=k+1}^\ell c_i -\sum_{E_i\subset P} (4 + E_i^2).
    \end{split}
\end{equation*}
To complete the proof, we will argue that 
\begin{equation}
\label{eqn:keyineq}
\sum_{i=k+1}^\ell c_i - \sum_{E_i\subset P} (4 + E_i^2) \geq 0 .
\end{equation}
To this end, note that since each $E_i\subset P$ is disjoint from other components of $\hat B$, we have that 
$$
0 > E_i^2 = E_i \cdot \hat B = E_i\cdot 2A
$$
is even.   Moreover, $E_i^2 = -1 - n_i$ where $n_i$ is the number of $j>i$ such that $x_j$ lies on the strict transform in \(T_{j-1}\) of $\sigma_i^{-1}(x_i)$.  Hence the only cause for concern occurs when $n_i = 1$ and $4+E_i^2 = 2$. 

But for each $E_i\subset P$, we can invoke Lemma \ref{lem:alternatives}.  In Case~\eqref{item:alternative-1} of that Lemma, we have that $4+E_i^2 \leq 1$ so \(-(4+E_i^2) \geq 0\).  In Case~\eqref{item:alternative-2}, we have that 
$$
4 + E_i^2 \leq 2 \leq c_j
$$
for some $j$ such that $x_j$ is immediately proximate to $x_i$.  Since $x_1,\dots,x_k$ are not proximate to any $x_i$ and, for each \(j>k\), $x_j$ is immediately proximate to exactly one $x_i$,
this shows that
\(\sum_{i=k+1}^\ell c_i - \sum_{E_i\subset P, \text{ Case~\eqref{item:alternative-2}}} (4+E_i^2) \geq 0\), so
the inequality \eqref{eqn:keyineq} is proved.
\end{proof}

As mentioned in the introduction, Koll\'ar proved nonexistence of double sections of conic bundles with low degree branch curve.
In particular, he proved that if the discriminant curve \(\Delta\) of a standard conic bundle \(X\to\mathbb P^2\) is general, then there is no surface \(S\subset X\) whose normalization maps to \(\mathbb P^2\) generically finitely with branch curve of degree \(<\frac{1}{2}\deg\Delta - 1\) \cite[Lemma 19]{Kollar17}. However, the following example shows that there exist rational double covers of \(\PP^2\) with branch curve of arbitrarily high degree.

\begin{exmp}[{\cite[Lemma 8.6]{Calabri}}]\label{exmp:high-multiplicity-rational-double-cover}
    Let \(d\geq 1\), and let \(B\subset\PP^2\) be a reduced curve of degree \(2d\) with a point \(x_0\) of multiplicity \(2d-1\) or \(2d-2\). Then the double cover of \(\PP^2\) branched along \(B\) is a normal rational surface.
\end{exmp}

The following example shows that the lower bound in Proposition~\ref{prop:5deg-bound} is sharp.
\begin{exmp}[{\cite[Proof of Proposition 8.7, case \((c^1)\)]{Calabri}}]\label{exmp:sharp}
    Let \(d\geq 2\), and let \(B\) be the plane curve \(M_1 + M_2 + L_1 + \cdots + L_{2d-2}\), where \(L_1,\ldots,L_{2d-2}\) are distinct lines passing through a common point \(x_0\), and \(M_1,M_2\) are distinct lines not containing \(x_0\) and such that the intersection point \(M_1\cap M_2\) is not contained in any \(L_i\). Then \(g(B)=0\), and by Example~\ref{exmp:high-multiplicity-rational-double-cover}, the double cover of \(\mathbb P^2\) branched along \(B\) is a normal rational surface.
    The singular points of \(B\) are \(x_0\) with multiplicity \(2d-2\); the \(4d-4\) intersection points \(M_i \cap L_j\), each with multiplicity 2; and \(M_1\cap M_2\) with multiplicity 2.
    So \(k=4d-2\) and \(\sum_{i=1}^k c_i = 10d-8\). For any integer \(e\geq 0\) we have \[e\deg B - \sum_{i=1}^k c_i - 2g(B) = 2d(e-5)+8.\] In particular \(5\deg B - \sum_{i=1}^k c_i - 2g(B) = 8\).
\end{exmp}
Example~\ref{exmp:sharp} also shows that for any \(e < 5\), the quantity \(e\deg B - \sum_{i=1}^k c_i - 2g(B)\) does not have a uniform lower bound among branch curves \(B\) of normal rational double covers of \(\mathbb P^2\).

\subsection{Birational classification of rational double covers of \texorpdfstring{\(\mathbb P^2\)}{P2}}

In this section, we recall a classification result that we will use in the next section.
Normal double covers $f_1\colon S_1\to\PP^2$ and $f_2\colon S_2\to \PP^2$ are \defi{birationally equivalent} if there are birational transformations $\tilde\varphi\colon S_1\tto S_2$ and $\varphi\colon \PP^2\tto \PP^2$ such that $f_2\circ \tilde\varphi = \varphi\circ f_1$.  In this case, the branch curve $B_\varphi$ of $f_2$ is (as before) the reduction mod $2$ of the total transform $\varphi_*B$ of the branch curve $B$ of $f_1$.  Since $\varphi$ and $\varphi^{-1}$ contract only rational curves, irreducible components of $B_\varphi$ with positive geometric genus correspond bijectively with the positive genus components of $B$.  In particular, the total genus of the branch curve is unchanged by $\varphi$.

The rational double covers of $\PP^2$ have been classified up to birational equivalence (see \cite{BayleBeauville,Calabri,CilibertoDoubleCovers}, though there is an unresolved technical issue in case \((a^2_2)\) of the proof of \cite[Proposition 9.4]{Calabri}).

\begin{thm}[{see \cite[Theorem 1.2]{CilibertoDoubleCovers} and \cite[Proposition 8.7]{Calabri}}]\label{thm:rational-double-plane-classification}
Suppose that $f\colon S\to\PP^2$ is a double cover by a normal projective surface $S$.  Then $S$ is rational if and only if there is a birational transformation $\varphi\colon \PP^2\tto \PP^2$ that equates $f$ to a double cover whose branch curve $B_\varphi$ is one of the following:
    \begin{enumerate}
        \item\label{item:calabri-thm-9.18-case-1} a smooth conic;
        \item\label{item:calabri-thm-9.18-case-2} a smooth quartic;
        \item\label{item:calabri-thm-9.18-case-3} an irreducible reduced sextic curve with two infinitely near triple points; or
        \item\label{item:calabri-thm-9.18-case-4} an irreducible reduced curve of degree $2d\geq 4$ with an ordinary singularity at a point of multiplicity $2d-2$ and at most one other singularity, which is an ordinary node.
    \end{enumerate}
In particular, if \(S\) is rational, then the branch curve of $f$ has at most one irreducible component with positive geometric genus.
\end{thm}
Let us record the total genus of $B_\varphi$ (i.e., of $B$) in each case.  We have \(g(B_\varphi)=0\) in case~\eqref{item:calabri-thm-9.18-case-1}, \(g(B_\varphi)=3\) in case~\eqref{item:calabri-thm-9.18-case-2}, \(g(B_\varphi)=4\) in case~\eqref{item:calabri-thm-9.18-case-3}, and \(g(B_\varphi)=2d-2\) or \(2d-3\) in case~\eqref{item:calabri-thm-9.18-case-4}.
Note that each case of \(B_\varphi\) in Theorem~\ref{thm:rational-double-plane-classification} satisfies \(5\deg (B_\varphi) - \sum_{i=1}^k c_i - 2g(B_\varphi) \geq 10\) (and in fact \(3\deg (B_\varphi) - \sum_{i=1}^k c_i - 2g(B_\varphi) \geq 6\)); however, this birational classification does not imply a bound of the type in Proposition~\ref{prop:5deg-bound} because in general one cannot control how these quantities change under arbitrary birational transformations of \(\PP^2\).

\section{Hyperbolicity and transversality} \label{sec:hyperbolicity}
Now we combine Proposition~\ref{prop:5deg-bound} with an algebraic hyperbolicity bound of Chen to prove that, under the assumptions of Theorem~\ref{thm:main-thm}, $B$ must meet $\Delta$ transversely.

\begin{prop}\label{prop:degree-geq-14-transverse-intersection}
    Let \(\Delta\) be a very general plane curve of degree \(\geq 18\) and \(\piDoubleCover\colon \XDoubleCover\to\mathbb P^2\) be a double cover by a normal rational surface \(\XDoubleCover\). Then \(\Delta\) is not a component of the branch curve $B$ of $\piDoubleCover$, and \(B\) meets $\Delta$ transversely at some smooth point of $B$.
\end{prop}

\begin{proof}
    We first show that $\Delta$ is not a component of $B$. Suppose to get a contradiction that $\Delta$ is a component of $B$. Note that the genus of $\Delta$ is at least 136, so by Theorem~\ref{thm:rational-double-plane-classification},
    we see that $\Delta$ is birational to case~\eqref{item:calabri-thm-9.18-case-4} and in particular must be a hyperelliptic curve. However, a smooth plane curve of degree at least 3 is never hyperelliptic, so we see that this case is impossible.

    It suffices to show that the number of set-theoretic intersections between $\Delta$ and $B$ exceeds half their scheme theoretic intersection number; i.e., that
    \[ \# (\Delta \cap B) - \frac{(\deg \Delta)(\deg B)}{2} > 0.\]

    Let \(s\) be the number of irreducible components of \(B\). For each irreducible component \(B_j\) of \(B\),
    let \(\nu_j\colon B_j^\nu\to\mathbb P^2\) be the normalization of \(B_j\). By \cite[Theorem 1.7]{Chen04} we have the inequality
    \begin{equation}\label{eqn:alg-hyperbolicity-inequality-each-irred-comp}
        \#(\nu_j^{-1}(\Delta)) \geq (\deg\Delta - 4)(\deg B_j) - 2g(B_j^\nu) + 2.
    \end{equation}
    Since the normalization \(\nu\colon B^\nu\to B \) is the disjoint union \(\bigsqcup_{j=1}^s B_j^\nu\), we have
    \begin{equation}\label{eqn:alg-hyperbolicity-inequality-sum}
        \begin{split}
            \#(\nu^{-1}(\Delta)) &= \sum_{j=1}^s \left(\#(\nu_j^{-1}(\Delta))\right) \\ 
            &\geq (\deg\Delta - 4) (\deg B) - 2g(B) + 2s \quad\text{by~\eqref{eqn:alg-hyperbolicity-inequality-each-irred-comp}} .
        \end{split}
    \end{equation}

    For a point \(x_i\in\mathbb P^2\), let \(c_i\) be the multiplicity of \(B\) at \(x_i\).
    The number of distinct intersection points of \(\Delta\) and \(B\) satisfies
    \begin{equation*}
        \begin{split}
            \#(\Delta\cap B) &\geq \#(\nu^{-1}(\Delta)) - \sum_{x_i \in \Delta\cap B} (c_i - 1) \\
            &\geq \#(\nu^{-1}(\Delta)) - \sum_{x_i\in B} (c_i-1) \geq \#(\nu^{-1}(\Delta)) - \sum_{x_i\in\Sing(B)} c_i.
        \end{split}
    \end{equation*}
    Combining this inequality with~\eqref{eqn:alg-hyperbolicity-inequality-sum} we get
    \begin{equation*}
        \begin{split}
            \# (\Delta \cap B)
            & \geq
            (\deg\Delta - 4) (\deg B) - 2g(B) + 2s - \sum_{x_i\in\Sing(B)} c_i \\
            &= (\deg\Delta - 9) (\deg B) + 5\deg B - \sum_{x_i\in\Sing(B)} c_i - 2 g(B) + 2s \\
            &\geq (\deg\Delta - 9) (\deg B) + 8 + 2s
        \end{split}
    \end{equation*}    
    where the last inequality holds by Proposition~\ref{prop:5deg-bound}.
    So
    \[\# (\Delta \cap B) - \frac{(\deg \Delta)(\deg B)}{2} \geq \left(\frac{1}{2}\deg\Delta - 9\right) (\deg B) + 8 + 2s .\]
    Since \(\deg\Delta \geq 18\), this quantity is strictly positive,
    as required.

\end{proof}

\begin{proof}[Proof of Theorem~\ref{thm:main-thm}]
    Let \(\XDoubleCover\to\mathbb P^2\) be a generically finite morphism of degree \(2\) from a rational surface.
    Let \(\XDoubleCover^\nu\) be the normalization of \(\XDoubleCover\), and let \(\XDoubleCover^\nu\to\bar{\XDoubleCover}\to\mathbb P^2\) be the Stein factorization. By Proposition~\ref{prop:degree-geq-14-transverse-intersection}, the branch locus $B$ of \(f\colon\bar{\XDoubleCover}\to\mathbb P^2\) meets \(\Delta\) transversely at some smooth point of $B$, so by Corollary~\ref{cor:deg-2-multisection-restriction}
    the base change \(X_{\bar{\XDoubleCover}}\to\bar{\XDoubleCover}\) does not admit a rational section.
\end{proof}

\bibliographystyle{plain}
\bibliography{main}

\end{document}